\documentclass[11pt,fleqn]{amsart}
\usepackage{amsmath,amsthm,listings,tikz,latexsym, hyperref}
\newtheorem{theorem}{Theorem}
\newtheorem{lemma}{Lemma}
\title{A Product Rule for Triangular Numbers}
\author{David G Radcliffe}
\email{dradcliffe@gmail.com}

\begin{document}

\begin{abstract}
We prove that there are exactly five sequences, including the triangular numbers,
that satisfy the product rule $T(mn) = T(m) T(n) + T(m-1) T(n-1)$ for all $m, n \ge 1$.
\end{abstract}

\maketitle

\section{Introduction}

The $n$th triangular number\cite{sloane} is 
$$T(n) = 1 + 2 + \ldots + n = \frac12 n(n+1).$$
It represents the number of dots in a triangular arrangement of dots
with 1 dot in the first row, 2 dots in the second row, and so on.

The triangular numbers satisfy many interesting properties, including a \emph{product rule}:
$$T(mn) = T(m) T(n) + T(m-1) T(n-1).$$
This rule can be demonstrated visually by subdividing a triangle into smaller triangles. Figure 1 illustrates the case 
$T(20) = T(5) T(4) + T(4) T(3)$.

The square numbers $S(n) = n^2$ satisfy the analogous, but simpler, product rule
$S(mn) = S(m) S(n)$. It would be interesting to explore product rules for other
sequences of figurate numbers.

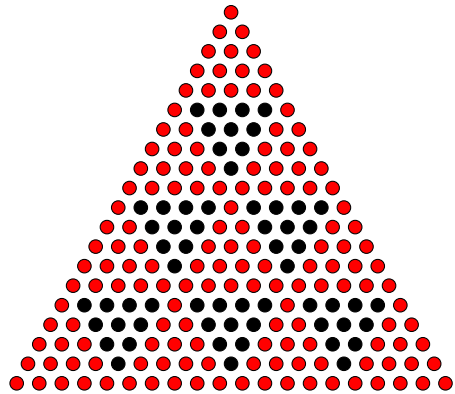
\begin{figure}
\centering
\begin{tikzpicture}[xscale=0.3,yscale=0.3]
  \foreach \x [evaluate=\x as \xmod using {int(mod(\x,5))}] in {0,...,19}
  {
    \foreach \y [evaluate=\y as \ymod using {int(mod(\y,5))}, evaluate=\ymod as \k using { \ymod > \xmod ? "black" : "red" }] in {0,...,\x}
    \draw [fill=\k] (\x - \y / 2, \y * 0.866) circle (0.3);
  }
\end{tikzpicture}
\caption{The 20th triangular number.}
\end{figure}

In this note, we determine all sequences of real numbers that satisfy the product rule for triangular numbers. This answers a question posed by Tanton\cite{Tanton}.
Our main result is the following:

\begin{theorem} \label{maintheorem} 
Each of the following sequences satisfies the product rule 
$$T(mn) = T(m) T(n) + T(m-1) T(n-1)$$ 
for all $m, n \ge 1$, and no other sequences satisfy this recurrence. 

\begin{enumerate}
\item $T(n) = 0$ for all $n \ge 0$.
\item $T(n) = \frac12$ for all $n \ge 0$.
\item $T(0) = 0$ and $T(2n) = T(2n-1) = n$ for all $n \ge 1$.
\item $T(3n) = T(3n+2) = 0$ and $T(3n+1) = 1$ for all $n \ge 0$.
\item $T(n) = \frac12 n(n+1)$ for all $n \ge 0$.
\end{enumerate}
\end{theorem}

Note that the fourth sequence is equivalent to the triangular numbers modulo 3.

\section{Verifying the solutions}
\label{verify}
In this section, we prove that all of the sequences listed in Theorem 1 
satisfy the product rule for triangular numbers. Readers who prefer to
verify the solutions for themselves should feel free to skip to the next
section.

\begin{lemma}
\label{constant-zero}
If $T(n) = 0$ for all $n \ge 0$, then $T$ satisfies the product rule.
\end{lemma}

\begin{proof}
If $m, n \ge 1$ then $T(mn) = 0$ and
\begin{align*}
T(m) T(n) + T(m-1) T(n-1) =  0\cdot 0 + 0\cdot 0 = 0.
\end{align*}
\end{proof}

\begin{lemma}
\label{constant-half}
If $T(n) = \frac12$ for all $n\ge0$, then $T$ satisfies the product rule.
\end{lemma}
\begin{proof}
If $m, n \ge 1$ then $T(mn) = \frac12$ and
\begin{align*}
T(m) T(n) + T(m-1) T(n-1) = \tfrac12 \cdot \tfrac12 + \tfrac12 \cdot \tfrac12 = \tfrac12.
\end{align*}
\end{proof}

\begin{lemma}
\label{half-n}
If $T(0) = 0$ and $T(2n) = T(2n-1) = n$ for all $n \ge 1$, then $T$ satisfies the product rule.
\end{lemma}

\begin{proof}
If $m=2a$ and $n=2b$ then $T(mn) = T(4ab) = 2ab$ and
\begin{align*}
T(m) T(n) + T(m-1) T(n-1) = ab + ab = 2ab.
\end{align*}

If $m=2a$ and $n=2b-1$ then $T(mn) = T(2a(2b-1)) = a(2b-1)$ and
\begin{align*}
T(m) T(n) + T(m-1) T(n-1) = ab + a(b-1) = a(2b-1).
\end{align*}

If $m=2a-1$ and $n=2b-1$, then 
\begin{align*}
 T(mn)  = T(4ab-2a-2b+1) & = 2ab - a - b + 1 \text{ and} \\
 T(m) T(n) + T(m-1) T(n-1) & = ab + (a-1)(b-1) \\
                                           & = 2ab - a - b + 1.
\end{align*}
\end{proof}

\begin{lemma}
\label{triangular-mod-3}
If $T(3n) = T(3n+2) = 0$ and $T(3n+1)=1$ for all $n \ge 0$, then 
$T$ satisfies the product rule.
\end{lemma}

\begin{proof}
Let $m, n \ge 1$. If $m \equiv 0 \pmod{3}$, then $T(mn) = 0$ and
\begin{align*}
& T(m) T(n) + T(m-1) T(n-1) = 0 \cdot T(n) + 0 \cdot T(n-1) = 0.
\end{align*}
The case where $n \equiv 0 \pmod{3}$ is similar.
 
If $m \equiv n \equiv 1 \pmod{3}$, then $T(mn) = 1$ and
\begin{align*}
& T(m) T(n) + T(m-1) T(n-1) = 1 \cdot 1 + 0 \cdot 0 = 1. 
\end{align*}

If $m \equiv n \equiv 2 \pmod{3}$, then $T(mn) = 1$ and
\begin{align*}
& T(m) T(n) + T(m-1) T(n-1) = 0 \cdot 0 + 1 \cdot 1 = 1.
\end{align*}

If $m \equiv 1 \pmod{3}$ and $n \equiv 2 \pmod{3}$, then $T(mn) = 0$ and
\begin{align*}
& T(m) T(n) + T(m-1) T(n-1) = 1 \cdot 0 + 0 \cdot 1 = 0.
\end{align*}
The case where $m \equiv 2 \pmod{3}$ and $n \equiv 1 \pmod{3}$ is similar.

\end{proof}

\begin{lemma}
\label{triangular}
If $T(n) = \frac12 n(n+1)$ for all $n \ge 0$ 
then $T$ satisfies the product rule.
\end{lemma}

\begin{proof}
Let $m, n \ge 1$. Then $T(mn) = \frac12 mn(mn+1)$ and
\begin{align*}
& T(m) T(n) + T(m-1) T(n-1)  \\
& = \frac12 m(m+1) \frac12 n(n+1) + \frac12 (m-1)m \frac12 (n-1)n \\
& = \frac{mn}{4} \left( (m+1)(n+1) + (m-1)(n-1) \right) \\
& = \frac{mn}{4} (2mn + 2)\\
&= \frac12 mn(mn+1).
\end{align*}
\end{proof}

\section{Proof of completeness}
We will prove that the list of solutions from the previous solutions is complete. Let $T$ be any sequence
that satisfies the product rule $$T(mn) = T(m) T(n) + T(m-1) T(n-1)$$ 
for all $m, n \ge 1$. For the sake of brevity, we will
set $a = T(0)$, $b = T(1)$, $c =T(2)$, and $d = T(3)$.

\begin{lemma}
\label{b-zero-or-one}
The equation $b = b^2 + a^2$ holds. In particular, if $a=0$ then $b=0$ or $b=1$.
\end{lemma}

\begin{proof}
Substituting $m=1$ and $n=1$ into the product rule gives $T(1) = T(1) T(1) + T(0) T(0)$, or $b = b^2 + a^2$.
If $a = 0$ then $b = b^2$, which implies that $b = 0$ or $b = 1$.
\end{proof}

\begin{lemma}
\label{must-be-zero}
The identity $T(n) = b T(n) + a T(n-1)$ holds for all $n \ge 1$.
If $a = 0$ and $b = 0$ then $T(n) = 0$ for all $n \ge 0$.
\end{lemma}
\begin{proof}
The identity is verified by substituting $m=1$ into the product rule.
It follows that if $a = b = 0$ then $T(n) = 0$ for all $n \ge 0$.
\end{proof}

\begin{lemma}
\label{must-be-half}
If $a \ne 0$ then $T(n) = 1/2$ for all $n\ge0$.
\end{lemma}
\begin{proof}
Since $b = b^2 + a^2$, it follows that $b \ne 1$, so we can write the equation as
$$\frac{a}{1-b} = \frac{b}{a}.$$
The equation $T(n) = bT(n) + aT(n-1)$ from Lemma \ref{must-be-zero} implies that
$$\frac{T(n)}{T(n-1)} = \frac{a}{1-b} = \frac{b}{a}.$$
Therefore, $T(n) = a r^n$ for all $n \ge 0$, where $r = b/a$.

Substituting $m=2$ into the product formula gives
\begin{align*}
T(2n) &= T(2) T(n) + T(1) T(n-1) \\
ar^{2n} &= ar^2 \cdot ar^n + ar \cdot ar^{n-1} \\
r^{n} &= a(r^2 + 1)
\end{align*}
Since the right side is independent of $n$, it follows that $r = 1$, hence $T(n) = a$ for all $n \ge 0$.
But if $b = b^2 + a^2$ and $a = b \ne 0$, then $a = \frac12$.
Therefore, $T(n) = \frac12$ for all $n \ge 0$.
\end{proof}

\begin{lemma}
\label{four-terms}
If $a = 0$ and $b = 1$, then $T(n)$ can be computed recursively for all $n \ge 3$ by means of the identities
$T(2n) = c T(n) + T(n-1)$ and $T(2n-1) = T(n) + (d-c) T(n-1)$. In particular, $T(n)$ is a function of $c$ and $d$
for all $n\ge 3$.
\end{lemma}
\begin{proof}
\begin{align*}
T(2n) &= T(2) T(n) + T(1) T(n-1) = c T(n) + T(n-1) \\
T(4) &= T(2) T(2) + T(1) T(1) = c^2 + 1 \\
T(4n) &= T(4) T(n) + T(3) T(n-1) = (c^2 + 1) T(n) + dT(n-1) \\
T(4n) &= T(2) T(2n) + T(1) T(2n-1) = cT(2n) + T(2n-1).
\end{align*}
Combining these equations yields
\begin{align*}
T(2n-1) &= T(4n) - cT(2n) \\
             &= (c^2+1) T(n) + d T(n-1) - c^2 T(n) - cT(n-1) \\
             &= T(n) + (d-c) T(n-1).
\end{align*}
Since $T(2n)$ and $T(2n-1)$ are linear combinations of previous terms for all $n \ge 2$,
and the coefficients are functions of $c$ and $d$, it follows that $T(n)$ is a function of $c$ and $d$
for all $n \ge 3$.
\end{proof}

\begin{lemma}
\label{three-terms}
If $a = 0$ and $b = 1$, then 
$$d = \frac{3c^3 + c}{c^2+2c-1}.$$
Consequently, $T(n)$ is uniquely determined for each $n \ge 2$ by the value of $c$ alone.
\end{lemma}

\begin{proof}
Using the formulas from Lemma \ref{four-terms} we calculate as follows:

\begin{align*}
T(4) &= c^2 + 1 \\
T(5) &= T(3) + (d-c) T(2) = d + (d - c) c = d + cd - c^2 \\
T(6) &= cT(3) + T(2) = cd + c \\
T(8) &= cT(4) + T(3) = c(c^2+1) + d = c^3 + c + d \\
T(9) &= T(3) T(3) + T(2) T(2) = d^2 + c^2 \\
T(18) &= T(3) T(6) + T(2) T(5) = d(cd + c) + c(d + cd - c^2) \\
T(18) &= cT(9) + T(8) = c(d^2 + c^2) + (c^3 + c + d)
\end{align*}

Equating the last two expressions for $T(18)$ yields
\begin{align*}
& cd^2 + c^2d + 2cd - c^3 = cd^2 + 2c^3 + c + d \\
& c^2d + 2cd - d = 3c^3 + c \\
& (c^2 + 2c - 1)d = 3c^3 + c
\end{align*}

If $c^2 + 2c - 1 = 0$ then $3c^3 + c = 0$ as well. But the polynomials have no roots in common,
which is a contradiction. Therefore, we may divide by $c^2 + 2c - 1$, yielding
$$d = \frac{3c^3 + c}{c^2+2c-1}.$$
This implies that the value of $T(n)$ is determined by $c$ alone.
\end{proof}

\begin{lemma}
\label{zero-one-or-three}
If $a=0$ and $b=1$ then $c \in \{0, 1, 3\}$.
\end{lemma}

\begin{proof}
By our previous results, we
may compute $T(n)$ recursively by the following formulas.
\begin{align*}
& T(0) = 0; T(1) = 1; T(2) = c \\
& T(3) = d = \frac{3c^3+c}{c^2+2c-1} \\
& T(2n) = cT(n) + T(n-1) \text{ for } n \ge 2 \\
& T(2n-1) = T(n) + (d-c) T(n-1) \text{ for } n \ge 3
\end{align*}

These formulas are implemented by the Python script in Figure 3.
The script uses SymPy\cite{SymPy}, a Python library for symbolic mathematics.
Since the last two expressions in the script must be zero, 
$c$ must satisfy the following equations:
\begin{align*}
 0 &= T(9) - T(3) T(3) - T(2) T(2) \\
    &= \frac{c(c - 3)(c - 1)(c + 1)(2c^3 + c - 1)}{(c^2 + 2c - 1)^2} \text{, and}\\
0  &= T(15) - T(3) T(5) - T(2) T(4) \\
    &= \frac{c(c - 3)(c - 1)(8c^6 - c^5 + 7c^4 - 4c^3 + 4c^2 - 3c + 1)}{(c^2+2c-1)^3}.
\end{align*}
The only solutions that are common to both equations are $c = 0, 1, 3$.
\end{proof}

\begin{figure}
\begin{lstlisting}[frame=single]
from sympy import Symbol, factor, simplify
c = Symbol('c')
d = (3*c**3 + c) / (c**2 + 2*c - 1)
T = [0, 1, c, d] + [0]*12
for n in range(4, 16):
    if n % 2 == 0:
        T[n] = simplify(c*T[n // 2] + T[n // 2 - 1])
    else:
        T[n] = simplify(T[n//2+1] + (d-c)*T[n//2])
print (factor(T[9] - T[3] * T[3] - T[2] * T[2]))
print (factor(T[15] - T[3] * T[5] - T[2] * T[4]))
\end{lstlisting}
\caption{Python script to compute $T(n)$ in terms of $c$.}
\end{figure}

\noindent {\bf Proof of Theorem \ref{maintheorem}.}
We know from Section \ref{verify} that all of the sequences listed in Theorem \ref*{maintheorem}
satisfy the product rule for triangular numbers.

Let $T$ be any sequence that satisfies the product rule for triangular numbers.
If $T(0) \ne 0$, then $T(n) = \frac12$ for all $n\ge0$ by Lemma \ref{must-be-half}.
If $T(0) = 0$, then either $T(1) = 0$ or $T(1) = 1$ by Lemma \ref{b-zero-or-one}.
In the first case, $T(n) = 0$ for all $n$ by Lemma \ref{must-be-zero}.

In the second case, $T(2) = 0$, $1$, or $3$ by Lemma \ref{zero-one-or-three}, 
and this value determines the value of $T(n)$ for all $n \ge 3$ by Lemma \ref{three-terms}.
If $T(2) = 0$, then $T(3n)=T(3n+2)=0$ and $T(3n)=1$ for all $n\ge0$ by Lemma \ref{triangular-mod-3}.
If $T(2) = 1$, then $T(2n) = T(2n-1) = n$ for all $n \ge 0$ by Lemma \ref{half-n}.
If $T(2) = 3$, then $T(n) = \frac12 n(n+1)$ for all $n \ge 0$ by Lemma \ref{triangular}.
These cases are mutually exclusive and collectively exhaustive, so the proof is complete.
See Figure \ref{flowchart}.
\hfill$\Box$

\usetikzlibrary{shapes.geometric, arrows}

\tikzstyle{startstop} = [rectangle, rounded corners, minimum width=3cm, minimum height=1cm,text centered, draw=black, fill=red!30]
\tikzstyle{process} = [rectangle, minimum width=3cm, minimum height=1cm, text centered, draw=black, fill=orange!30]
\tikzstyle{decision} = [diamond, minimum width=3cm, minimum height=1cm, text centered, draw=black, fill=green!30]
\tikzstyle{arrow} = [thick,->,>=stealth]

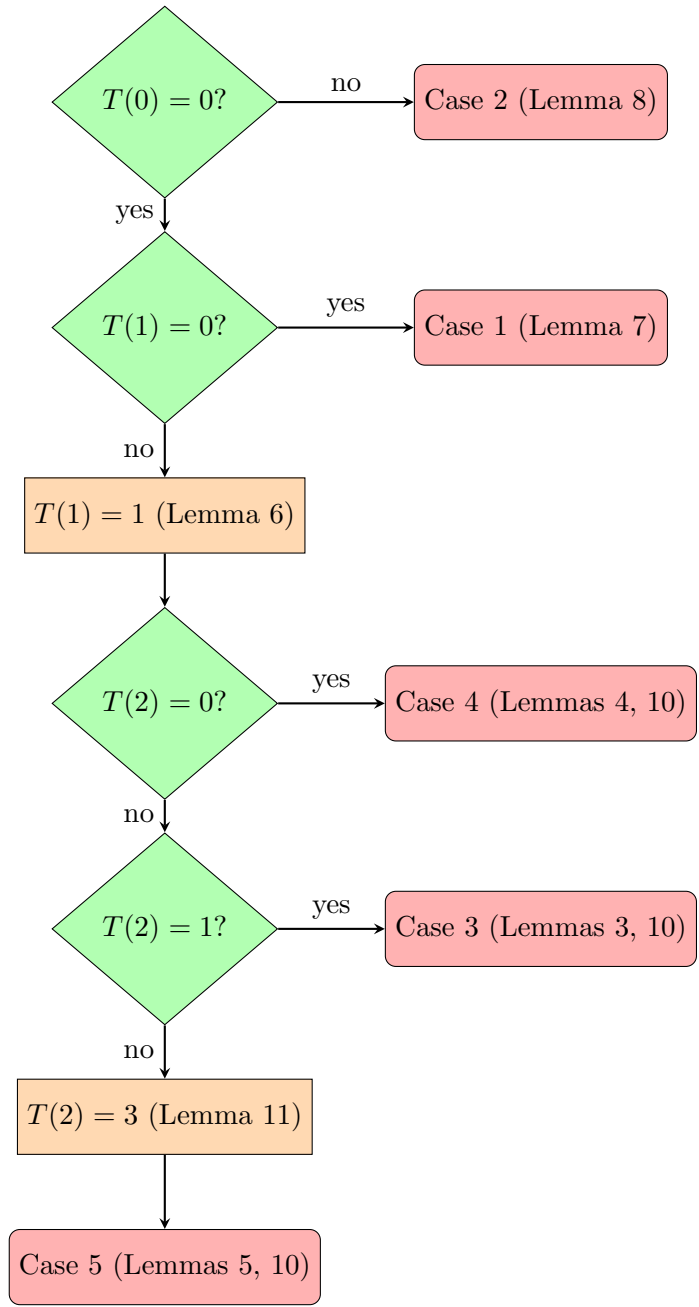
\begin{figure}
\label{flowchart}
\centering
\begin{tikzpicture}[node distance=2cm]
\node (isa0) [decision] {$T(0) = 0$?};
\node (case2) [startstop, right of=isa0, xshift=3cm] {Case 2 (Lemma \ref*{must-be-half})};
\node (isb0) [decision, below of=isa0, yshift=-1cm] {$T(1) = 0$?};
\node (case1) [startstop, right of=isb0, xshift=3cm] {Case 1 (Lemma \ref*{must-be-zero})};
\node (bis1) [process, below of=isb0, yshift=-0.5cm] {$T(1) = 1$ (Lemma \ref*{b-zero-or-one})};
\node (isc0) [decision, below of=bis1, yshift=-0.5cm] {$T(2) = 0$?};
\node (case4) [startstop, right of=isc0, xshift=3cm] {Case 4 (Lemmas \ref*{triangular-mod-3}, \ref*{three-terms})};
\node (isc1) [decision, below of=isc0, yshift=-1cm] {$T(2) = 1$?};
\node (case3) [startstop, right of=isc1, xshift=3cm] {Case 3 (Lemmas \ref*{half-n}, \ref*{three-terms})};
\node (cis3) [process, below of=isc1, yshift=-0.5cm] {$T(2) = 3$ (Lemma \ref*{zero-one-or-three})};
\node (case5) [startstop, below of=cis3] {Case 5 (Lemmas \ref*{triangular}, \ref*{three-terms})};

\draw [arrow] (isa0) -- node[anchor=south] {no} (case2);
\draw [arrow] (isa0) -- node[anchor=east] {yes} (isb0);
\draw [arrow] (isb0) -- node[anchor=south] {yes} (case1);
\draw [arrow] (isb0) -- node[anchor=east] {no} (bis1);
\draw [arrow] (bis1) -- (isc0);
\draw [arrow] (isc0) -- node[anchor=south] {yes} (case4);
\draw [arrow] (isc0) -- node[anchor=east] {no} (isc1);
\draw [arrow] (isc1) -- node[anchor=south] {yes}(case3);
\draw [arrow] (isc1) -- node[anchor=east] {no} (cis3);
\draw [arrow] (cis3) -- (case5);
\end{tikzpicture}
\caption{Flowchart for the proof of Theorem \ref*{maintheorem}.}
\end{figure}

\bibliography{mybib}{}
\bibliographystyle{plain}

\end{document}